\documentclass[12pt,reqno]{amsart}
\usepackage{amssymb}
\usepackage{amsxtra}
\usepackage{mathrsfs}
\usepackage[all]{xy}
\usepackage{cite}
\usepackage[unicode,hypertex]{hyperref}
\newtheorem{theorem}{Theorem}

\newtheorem{prop}[theorem]{Proposition}

\theoremstyle{definition}

\theoremstyle{remark}
\newtheorem{example}{Example}[section]

\newcommand*{\Ptens}{\mathop{\widehat\otimes}}

\newcommand*{\Tens}{\mathop{\otimes}}
\newcommand*{\Smash}{\mathop{\#}}
\newcommand*{\Psmash}{\mathop{\widehat{\#}}}

\newcommand*{\id}{1}

\newcommand*{\wh}{\widehat}
\newcommand*{\wt}{\widetilde}

\newcommand*{\la}{\langle}
\newcommand*{\ra}{\rangle}

\newcommand*{\CC}{\mathbb C}
\newcommand*{\N}{\mathbb N}
\newcommand*{\Z}{\mathbb Z}

\newcommand*{\cO}{\mathscr O}
\newcommand*{\fg}{\mathfrak g}

\newcommand*{\eps}{\varepsilon}
\newcommand*{\xra}{\xrightarrow}
\begin{document}
\title[The Arens-Michael envelope of a smash product]{The Arens-Michael
envelope\\ of a smash product}
\author{A. Yu. Pirkovskii}
\address{Department of Geometry and Topology,
Faculty of Mathematics\\
State University -- Higher School of Economics\\
Vavilova 7, 117312 Moscow, Russia}
\email{aupirkovskii@hse.ru, pirkosha@online.ru}
\thanks{This work was partially supported by the RFBR grant 08-01-00867,
by the Ministry of Education and Science of Russia (programme ``Development of the scientific
potential of the Higher School'', grant no. 2.1.1/2775),
and by the President of Russia grant MK-1173.2009.1.}
\date{}
\begin{abstract}
Given a Hopf algebra $H$ and an $H$-module
algebra $A$, we explicitly describe the Arens-Michael envelope of
the smash product $A\Smash H$ in terms of the Arens-Michael envelope of $H$
and a certain completion of $A$.
We also give an example (Manin's quantum plane) showing that the result
fails for non-Hopf bialgebras.
\end{abstract}
\maketitle

\section{Introduction}

The Arens-Michael envelope \cite{T1,X2} of an associative $\CC$-algebra $A$ is the
completion of $A$ with respect to the family of all submultiplicative seminorms on $A$.
For example \cite{T2}, the Arens-Michael envelope of the polynomial algebra
$\CC[t_1,\ldots ,t_n]$ is the algebra of holomorphic functions on $\CC^n$.
More generally \cite{Pir_qfree}, if $A$ is the algebra
of regular (i.e., polynomial) functions on a complex affine algebraic variety $V$, then
the Arens-Michael envelope of $A$ is the algebra of holomorphic functions on $V$.
This result suggests that the Arens-Michael envelope of a ``quantized polynomial
algebra'' (see, e.g., \cite{Art,Good_surv}) can be viewed as a
``quantized algebra of holomorphic functions''. From this point of view,
Arens-Michael envelopes can be potentially useful for the
development of noncommutative complex analytic geometry.
For further information on Arens-Michael envelopes, we refer
to \cite{Pir_qfree,Pir_stb,Pir_smash}.

In this short note, we extend our earlier result obtained in \cite{Pir_smash}.
Let $H$ be a Hopf algebra, let $A$ be an $H$-module algebra, and let
$A\Smash H$ denote the smash product of $A$ by $H$.
Assuming that $H$ is cocommutative, we proved
\cite[Theorem 2.2]{Pir_smash} that the Arens-Michael envelope of
$A\Smash H$ is isomorphic to the analytic smash product of
the ``$H$-completion'' of $A$ by
the Arens-Michael envelope of $H$.
Here our goal is to show that the result holds without
the cocommutativity assumption, but fails for non-Hopf bialgebras.

\section{Preliminaries}

We shall work over the complex numbers $\CC$. All associative algebras
and algebra homomorphisms are assumed to be unital.
Modules over algebras are also assumed to be unital
(i.e., $1\cdot x=x$ for each left $A$-module $X$ and for each $x\in X$).

By a topological algebra we mean a topological vector space $A$
together with the structure of an associative algebra such that
the multiplication map $A\times A\to A$ is separately continuous.
A complete, Hausdorff, locally convex topological algebra with
jointly continuous multiplication is called a {\em $\Ptens$-algebra} \cite{T1,X1}.
If $A$ is a $\Ptens$-algebra, then
the multiplication $A\times A\to A$ extends
to a continuous linear map from the completed projective tensor
product $A\Ptens A$ to $A$. In other words, a
$\Ptens$-algebra is just an
algebra in the tensor category $(\mathbf{LCS},\Ptens)$
of complete Hausdorff locally convex spaces.
This observation can be used to define $\Ptens$-coalgebras,
$\Ptens$-bialgebras, and Hopf $\Ptens$-algebras; see, e.g., \cite{BFGP}.

If $A$ is a $\Ptens$-algebra, then a {\em left $A$-$\Ptens$-module}
is a complete, Hausdorff locally convex space $X$ together with the
structure of a left $A$-module such that the action
$A\times X\to X$ is jointly continuous. Right $A$-$\Ptens$-modules
and $A$-$\Ptens$-bimodules are defined similarly.

Recall that a seminorm $\|\cdot\|$ on an algebra $A$ is
{\em submultiplicative} if $\| ab\|\le\| a\|\| b\|$ for all $a,b\in A$.
A topological algebra $A$ is said to be {\em locally $m$-convex}
if its topology can be defined by a family of submultiplicative seminorms.
Note that the multiplication in a locally $m$-convex algebra is
jointly continuous.
An {\em Arens-Michael algebra} is a complete, Hausdorff,
locally $m$-convex algebra.

Let $A$ be a topological algebra. A pair $(\wh{A},\iota_A)$ consisting of
an Arens-Michael algebra $\wh{A}$ and a continuous homomorphism
$\iota_A\colon A\to\wh{A}$ is called {\em the Arens-Michael envelope} of $A$
\cite{T1,X2}
if for each Arens-Michael algebra $B$ and for each continuous
homomorphism $\varphi\colon A\to B$ there exists a unique
continuous homomorphism $\wh{\varphi}\colon\wh{A}\to B$ making
the following diagram commutative:
\begin{equation*}
\xymatrix{
\wh{A} \ar@{-->}[r]^{\wh{\varphi}} & B \\
A \ar[u]^{\iota_A} \ar[ur]_\varphi
}
\end{equation*}
The Arens-Michael envelope always exists and
can be obtained as the completion\footnote[1]{Here we follow the convention
that the completion of a non-Hausdorff locally convex space $E$
is defined to be the completion of the associated Hausdorff
space $E/\overline{\{ 0\}}$.} of $A$
with respect to the family of all continuous submultiplicative seminorms on $A$
(see \cite{T1} and \cite[Chap. V]{X2}).
This implies, in particular, that $\iota_A\colon A\to\wh{A}$ has dense range.
Clearly, the Arens-Michael envelope is unique in the obvious sense.

Each associative algebra $A$ becomes a topological algebra
with respect to the strongest locally convex topology. The Arens-Michael
envelope, $\wh{A}$, of the resulting topological algebra will be referred
to as the Arens-Michael envelope of $A$. That is, $\wh{A}$ is
the completion of $A$ with respect to the family of {\em all} submultiplicative
seminorms.

If $H$ is a bialgebra (respectively, a Hopf algebra), then it is easy
to show that $\wh{H}$ is a $\Ptens$-bialgebra
(respectively, a Hopf $\Ptens$-algebra) in a natural way (for details, see
\cite{Pir_stb}).

In what follows, we will use standard notation from Hopf algebra theory.
In particular, given a Hopf algebra $H$, the symbols $\mu_H$, $\Delta_H$,
$\eta_H$, $\eps_H$, $S_H$ will denote the multiplication, the comultiplication,
the unit, the counit, and the antipode, respectively.
We will often suppress the subscript ``$H$'', when no
confusion is possible.

Let $H$ be a bialgebra. Recall that an {\em $H$-module algebra} is
an algebra $A$ endowed with the structure of a left $H$-module such that
the product $\mu_A\colon A\Tens A\to A$ and the unit map
$\eta_A\colon \CC\to A$ are $H$-module morphisms.
For example, if $\fg$ is a Lie algebra acting on $A$ by derivations,
then the action $\fg\times A\to A$ extends to a map $U(\fg)\times A\to A$
making $A$ into a $U(\fg)$-module algebra. Similarly, if $G$ is a semigroup
acting on $A$ by endomorphisms, then $A$ becomes a $\CC G$-module algebra,
where $\CC G$ denotes the semigroup algebra of $G$.

Given an $H$-module algebra $A$, the {\em smash product algebra} $A\Smash H$
is defined as follows (see, e.g., \cite{Sweedler}).
As a vector space, $A\Smash H$ is equal to $A\Tens H$.
To define multiplication,
denote by $\mu_{H,A}\colon H\Tens A\to A$ the action of $H$ on $A$,
and define $\tau\colon H\Tens A\to A\Tens H$ as the composition
\begin{equation}
\label{tau}
H\Tens A \xra{\Delta_H\otimes\id_A} H\Tens H\Tens A \xra{\id_H\otimes c_{H,A}}
H\Tens A\Tens H \xra{\mu_{H,A}\otimes\id_H} A\Tens H
\end{equation}
(here $c_{H,A}$ denotes the flip $H\Tens A\to A\Tens H$).
Using Sweedler's notation, we have
\begin{equation}
\label{tau_Sw}
\tau(h\otimes a)=\sum_{(h)} h_{(1)}\cdot a\otimes h_{(2)}.
\end{equation}
Then the map
\begin{equation}
\label{smashprod}
(A\Tens H)\Tens (A\Tens H) \xra{\id_A\otimes\tau\otimes\id_H}
A\Tens A\Tens H\Tens H \xra{\mu_A\otimes\mu_H} A\Tens H
\end{equation}
is an associative multiplication on $A\Tens H$. The resulting algebra
is denoted by $A\Smash H$ and is called the {\em smash product} of $A$ with $H$.
Using \eqref{tau_Sw}, we see that the multiplication on $A\Smash H$ is given by
\begin{equation}
\label{smashprod_Sw}
(a\otimes h)(a'\otimes h')=\sum_{(h)} a (h_{(1)}\cdot a')\otimes h_{(2)}h'.
\end{equation}
In particular, we have
\begin{align}
\label{h=1}
(a\otimes 1)(a'\otimes h')&=aa'\otimes h',\\
\label{a'=1}
(a\otimes h)(1\otimes h')&=a\otimes hh',\\
\label{tau_prod}
(1\otimes h)(a\otimes 1)&=\tau(h\otimes a).
\end{align}
This implies that $A$ and $H$ become subalgebras of $A\Smash H$ via the maps
$a\mapsto a\otimes 1$ and $h\mapsto 1\otimes h$.

Similar definitions apply in the $\Ptens$-algebra case.
Namely, if $H$ is a $\Ptens$-bialgebra, then an {\em $H$-$\Ptens$-module algebra}
is a $\Ptens$-algebra $A$ together with the structure of
a left $H$-$\Ptens$-module such that the product $A\Ptens A\to A$
and the unit map $\CC\to A$ are $H$-module morphisms.
By replacing $\Tens$ with $\Ptens$
in \eqref{tau} and \eqref{smashprod}, we obtain an associative, jointly
continuous multiplication
on $A\Ptens H$. The resulting $\Ptens$-algebra is denoted by $A\Psmash H$
and is called the {\em analytic smash product} of $A$ with $H$.

Let $H$ be an algebra, and let $A$ be a left $H$-module. We say that a seminorm
$\|\cdot\|$ on $A$ is {\em $H$-stable} \cite{Pir_smash} if for each $h\in H$
there exists $C>0$ such that $\| h\cdot a\|\le C\| a\|$ for each $a\in A$.
If $H$ is a bialgebra and $A$ is an $H$-module algebra, then
the {\em $H$-completion} of $A$ is the
completion of $A$ with respect to the family of all $H$-stable, submultiplicative
seminorms. The $H$-completion of $A$ will be denoted by $\wt{A}$.
It is immediate from the definition that $\wt{A}$ is an Arens-Michael algebra.

\begin{prop}[{\cite[Proposition 2.1]{Pir_smash}}]
\label{prop:Psmash_AM}
Let $H$ be a bialgebra, and let $A$ be an $H$-module algebra.
Then the action of $H$ on $A$ uniquely extends to an action of $\wh{H}$
on $\wt{A}$, so that $\wt{A}$ becomes an $\wh{H}$-$\Ptens$-module algebra.
Moreover, the smash product $\wt{A}\Psmash\wh{H}$ is an Arens-Michael algebra.
\end{prop}

\section{The results}

\begin{theorem}
\label{thm:main}
Let $H$ be a Hopf algebra, and let $A$ be an $H$-module algebra.
Then the canonical map $A\Smash H\to \wt{A}\Psmash\wh{H}$
extends to a $\Ptens$-algebra isomorphism
\begin{equation*}
(A\Smash H)\sphat\, \cong \wt{A}\Psmash\wh{H}.
\end{equation*}
\end{theorem}
\begin{proof}
Let $\varphi\colon A\Smash H\to B$ be a homomorphism to an
Arens-Michael algebra $B$. We endow $A$ and $H$ with
the topologies inherited from $\wt{A}$ and $\wh{H}$, respectively.
Since the canonical image of $A\Smash H$
is dense in $\wt{A}\Psmash\wh{H}$, it suffices to show that $\varphi$
is continuous with respect to the projective tensor product topology on $A\Smash H$.

Define $\varphi_1\colon A\to B$
and $\varphi_2\colon H\to B$ by $\varphi_1(a)=\varphi(a\otimes 1)$
and $\varphi_2(h)=\varphi(1\otimes h)$.
Clearly, $\varphi_1$ and $\varphi_2$ are algebra homomorphisms.
Using \eqref{h=1}, we have
\begin{equation*}
\varphi(a\otimes h)=
\varphi\bigl((a\otimes 1)(1\otimes h)\bigr)=
\varphi_1(a)\varphi_2(h)
\end{equation*}
for each $a\in A,\; h\in H$. Therefore we need only prove that $\varphi_1$
and $\varphi_2$ are continuous.

Let $\|\cdot\|$ be a continuous submultiplicative seminorm on $B$.
Then the seminorms $a\mapsto \| a\|'=\| \varphi_1(a)\|\; (a\in A)$
and $h\mapsto \| h\|''=\| \varphi_2(h)\|\; (h\in H)$
are submultiplicative. This implies, in particular, that $\varphi_2$ is continuous.
To prove the continuity of $\varphi_1$, we have to show that $\|\cdot\|'$
is $H$-stable.

For each $h\in H,\; a\in A$ we have the following identities in $A\Smash H$:
\begin{align*}
h\cdot a\otimes 1
&=\sum_{(h)} h_{(1)}\cdot a\otimes\eps(h_{(2)})1\\
&=\sum_{(h)} h_{(1)}\cdot a\otimes h_{(2)}S(h_{(3)})\\
&=\sum_{(h)} \tau(h_{(1)}\otimes a)(1\otimes S(h_{(2)}))
&& \text{by \eqref{tau_Sw} and \eqref{a'=1}}\\
&=\sum_{(h)} (1\otimes h_{(1)})(a\otimes 1)(1\otimes S(h_{(2)}))
&& \text{by \eqref{tau_prod}}.
\end{align*}
Therefore
\[
\begin{split}
\| h\cdot a\|'
&=\|\varphi_1(h\cdot a)\|
=\|\varphi(h\cdot a\otimes 1)\|\\
&=\Bigl\|\sum_{(h)} \varphi_2(h_{(1)}) \varphi_1(a) \varphi_2(S(h_{(2)}))\Bigr\|\\
&\le\sum_{(h)} \| h_{(1)}\|'' \| a\|' \| S(h_{(2)})\|''
=C\| a\|',
\end{split}
\]
where $C=\sum_{(h)} \| h_{(1)}\|'' \| S(h_{(2)})\|''$.
Thus $\|\cdot\|'$ is
$H$-stable, and so $\varphi_1$ is continuous.
In view of the above remarks,
$\varphi$ is also continuous, and so it uniquely extends to a $\Ptens$-algebra
homomorphism $\wt{A}\Psmash\wh{H}\to B$. This completes the proof.
\end{proof}

\begin{example}
It is natural to ask whether Theorem~\ref{thm:main} holds in the more general situation
where $H$ is a bialgebra. The following example shows that the answer is negative.
Let $A=\CC[x]$ be the polynomial algebra, and let the additive semigroup $\Z_+$
act on $A$ by
\[
(k\cdot f)(x)=f(q^{-k}x)\quad (f\in\CC[x],\; k\in\Z_+),
\]
where $q\in\CC\setminus\{ 0\}$ is a fixed constant.
Then $A$ becomes an $H$-module algebra, where $H=\CC\Z_+$ is the semigroup algebra
of $\Z_+$.

Given $k\in\Z_+$, let us write $\delta_k$ for the corresponding element of $H$.
If we identify $H$ with the polynomial algebra $\CC[y]$ by sending
the generator $\delta_1\in H$ to $y$,
then we obtain a vector space isomorphism
$A\Smash H\cong\CC[x,y]$. A straightforward computation shows that
the resulting multiplication on $\CC[x,y]$ is given by the formula
$xy=qyx$. In other words, we can identify $A\Smash H$ with
Manin's quantum plane \cite{Manin}
\[
\CC_q[x,y]=\CC\bigl\la x,y \,\bigl|\, xy=qyx\bigr\ra.
\]

Suppose now that $|q|<1$, and let $\|\cdot\|$ be an $H$-stable seminorm on $A$.
Then there exists $C>0$ such that $\| \delta_1\cdot f\|\le C\| f\|$ for all $f\in A$.
Setting $f=x^n$ and using the relation $\delta_1\cdot x^n=q^{-n} x^n$, we see that
\[
|q|^{-n} \| x^n\| \le C\| x^n\| \quad (n\in\Z_+).
\]
Since $|q|<1$, we conclude that there exists $N\in\N$ such that
$\| x^n\|=0$ for $n>N$.

It is easy to see that each seminorm of the form
\[
\| a\|_N=\sum_{i=0}^N |c_i| \quad \bigl(a=\sum c_i x^i\in A\bigr)
\]
is submultiplicative and $H$-stable. Moreover,
it follows from the above remarks that each $H$-stable seminorm on $A$ is dominated
by $\|\cdot\|_N$ for some $N$. Therefore the $H$-completion $\wt{A}$
is the completion of $A$ with respect to the topology generated by the
seminorms $\|\cdot\|_N,\; N\in\N$. Thus $\wt{A}$ can be identified with the algebra
$\CC[[x]]$ of formal power series endowed with the topology of coordinatewise
convergence. Since $\wh{H}$ is isomorphic to the algebra of entire functions
$\cO(\CC)$ \cite{T2}, we can identify the underlying
topological vector space of $\wt{A}\Psmash\wh{H}$ with
\begin{equation}
\label{formal_entire}
\CC[[x]]\Ptens\cO(\CC)
\cong
\Bigl\{ a=\sum_{i,j\in\Z_+} c_{ij} x^i y^j :
\| a\|_{\rho,N}=\sum_{i=0}^N \sum_{j=0}^\infty |c_{ij}|\rho^j<\infty\;\forall\rho>0\Bigr\}.
\end{equation}
On the other hand (see \cite[Corollary 5.14]{Pir_qfree}), the Arens-Michael envelope
of the quantum plane $\CC_q[x,y]$ (where $|q|<1$) can be identified with
\[
\Bigl\{ a=\sum_{i,j\in\Z_+} c_{ij} x^i y^j :
\| a\|_{\rho}=\sum_{i,j=0}^\infty |c_{ij}| |q|^{ij} \rho^{i+j}<\infty\;\forall\rho>0\Bigr\}.
\]
Comparing this with \eqref{formal_entire}, we see that the canonical map
$(A\Smash H)\sphat\to\wt{A}\Psmash\wh{H}$ (which always exists by the very
definition of the Arens-Michael envelope and by Proposition~\ref{prop:Psmash_AM})
is not onto. Thus Theorem~\ref{thm:main} cannot be generalized to non-Hopf
bialgebras.
\end{example}

\end{document}